\documentclass[12pt]{amsart}
\usepackage{amssymb}
\usepackage[mathscr]{eucal}
\usepackage{epsf}
\usepackage{epsfig}
\usepackage{color}
\DeclareGraphicsExtensions{.pstex,.eps,.epsi,.ps}

 \newtheorem{thm}{Theorem}[section]
 \newtheorem{cor}[thm]{Corollary}
 \newtheorem{lem}[thm]{Lemma}
 \newtheorem{prop}[thm]{Proposition}
 \theoremstyle{definition}
 
 \theoremstyle{remark}
 \newtheorem{rem}[thm]{Remark}
 \numberwithin{equation}{section}

 \newcommand{\Real}{\mathbb{R}}

 \newcommand{\hor}{\text{hor}}
 \newcommand{\ver}{\text{ver}}
 \newcommand{\tr}{\textbf{tr}}

 \newcommand{\ric}{\textbf{Rc}}

 \newcommand{\Rm}{\textbf{Rm}}

  \newcommand{\Rb}{\textbf{w}}
  \newcommand{\Ta}{\frac{\bar D}{dt}}
  
    \newcommand{\Lev}[1]{\frac{D}{d#1}}

\begin{document}

\title[Differential Harnack inequalities on Sasakian manifolds]{Differential Harnack inequalities for a family of sub-elliptic diffusion equations on Sasakian manifolds}

\author{Paul W.Y. Lee}
\email{wylee@math.cuhk.edu.hk}
\address{Room 216, Lady Shaw Building, The Chinese University of Hong Kong, Shatin, Hong Kong}

\thanks{The author's research was supported by RGC grant 404512.}

\date{\today}

\maketitle

\begin{abstract}
We prove a version of differential Harnack inequality for a family of sub-elliptic diffusions on Sasakian manifolds under certain curvature conditions.
\end{abstract}


\section{Introduction}

Harnack inequality is one of the most fundamental results in the theory of elliptic and parabolic equations. For linear parabolic equations in divergence form, this was first done in \cite{Mos}. Since then, numerous developments around this inequality were found. In \cite{LiYa}, the, so called, Li-Yau estimate was proved. This is a sharp gradient estimate for linear parabolic equations on Riemannian manifolds with a lower bound on the Ricci curvature. This estimate is also called a differential Harnack inequality since one can recover the Harnack inequality by integrating this estimate along geodesics.

There are many generalizations of the Li-Yau estimate for geometric evolution equations. This includes the evolution equations for hypersurfaces \cite{Ha3,Ch1,An}, the Yamabe flow \cite{Ch2}, the Ricci flow \cite{Ha2} and its K\"ahler analogue \cite{Ca}. For a more detail account of these generalizations
as well as further developments, see \cite{Ni}.

There are also generalizations \cite{BaLe} of the Li-Yau estimate to linear parabolic equations of the form
\begin{equation}\label{parabolic}
\dot \rho_t =L\rho_t
\end{equation}
under certain conditions called curvature-dimension conditions. Here $L$ is a second order linear elliptic operator without constant term. The curvature-dimension conditions were recently generalized by \cite{BaGa} to obtain Li-Yau type estimates for equations of the form (\ref{parabolic}), where $L$ is a linear sub-elliptic operator without constant term. The following is one of the main results in \cite{BaGa} when the underlying manifold is Sasakian and the equation is the sub-elliptic heat equation (for the definition of Sasakian manifolds and various related notions appeared in Theorem \ref{BaGaSasakian}, see Section 2).

\begin{thm}\label{BaGaSasakian}\cite{BaGa}
Assume that the manifold is Sasakian and satisfies $\overline\ric\geq 0$. Then any positive solution of the sub-elliptic heat equation
\[
\dot \rho_t=\Delta_\hor\rho_t
\]
satisfies
\[
\begin{split}
\left(1+\frac{3}{n}\right)\dot f_t+\frac{1}{2}|\nabla_\hor f_t|^2+&\frac{tn}{3}|\nabla_\ver f_t(x)|^2\leq\frac{2n\left(1+\frac{3}{n}\right)^2}{t}
\end{split}
\]
for all $t\geq 0$ and all $x$ in $M$, where $f_t=-2\log\rho_t$.
\end{thm}
We remark that a local version of the above estimate, which did not take the curvature into account, appeared in an earlier work \cite{CaYa}.

On the other hand, the author proved in \cite{Le1} another version of the differential Harnack inequality which is different from that of \cite{BaLe}.
\begin{thm}\label{Lee1}\cite{Le1}
Let $\rho_t$ be a positive solution of the equation
\[
\dot\rho_t=\Delta\rho_t+\left<\nabla\rho_t,\nabla U_1\right>+U_2\rho_t
\]
on a compact Riemannian manifold of non-negative Ricci curvature. Assume that
\[
\Delta\left(-\Delta U_1-\frac{1}{2}|\nabla U_1|^2+2U_2\right)\geq k_3.
\]
Then
\[
2\Delta\log\rho_t+\Delta U_1 \geq -n a_{\frac{k_3}{n}}(t),
\]
where \[
a_K(t)= \begin{cases}
\sqrt{K}\cot(\sqrt{K}\,t) & \mbox{if $K>0$}\\
\frac{1}{t} & \mbox{if $K= 0$}\\
\sqrt{-K}\coth(\sqrt{-K}\,t) & \mbox{if $K<0$}.
\end{cases}
\]
\end{thm}

In this paper, we combine the ideas from \cite{BaGa} and \cite{Le1} to obtain a differential Harnack estimate for (\ref{parabolic}), where $L$ is a linear sub-elliptic operator (possibly with constant term).
In fact, we allow $L$ to have a mild non-linearity (see Theorem \ref{keythm} for the detail). In the case when the manifold is Sasakian and $L$ is linear, we have the following result.

\begin{thm}\label{keythmSasakian}
Assume that the manifold is a compact Sasakian and it satisfies $\overline\ric\geq 0$. Let $U_1$ and $U_2$ be two smooth functions on $M$ satisfying
\begin{enumerate}
\item $V=\Delta_\hor U_1+\frac{1}{2}|\nabla_\hor U_1|^2-2U_2$
\item $V\leq \kappa_1$,
\item $\Delta_\hor V+\frac{n^2}{3\kappa_2}\left(1+\frac{3}{n}\right)^2|\nabla_\ver V|^2\leq \kappa_2$,
\end{enumerate}
for some positive constants $\kappa_1$ and $\kappa_2$. Then any positive solution of the equation
\[
\dot \rho_t=\Delta_\hor\rho_t+\left<\nabla_\hor U_1,\nabla_\hor\rho_t\right>+U_2\rho_t
\]
satisfies
\[
\begin{split}
&\left(1+\frac{3}{n}\right)\dot f_t(x)+\frac{1}{2}|\nabla_\hor f_t|^2_x-V(x)\\
&+\left(1+\frac{n}{3}\right)\sqrt{\frac{n}{2\kappa_2}}\tanh(c_2t)|\nabla_\ver f_t(x)|^2\leq \frac{\kappa_2}{c_2}\coth(c_2t)+\frac{3\kappa_1}{n}
\end{split}
\]
for all $t\geq 0$ and all $x$ in $M$, where $c_2=\frac{1}{n+3}\sqrt{\frac{n\kappa_2}{2}}$ and $f_t=-2\log\rho_t-U_1$.
\end{thm}

As usual, we can integrate the estimate in Theorem \ref{keythmSasakian} and obtain a Harnack estimate. Let
\[
W(x)=\left(\frac{n}{n+3}\right)^2V(x)+\frac{3nK_4}{\left(n+3\right)^2}.
\]
In this case, we consider the following cost function.
\[
c_{s_0,s_1}(x_0,x_1)=\inf_I\int_{s_0}^{s_1}\frac{1}{2}|\dot\gamma(s)|^2+W(\gamma(s))ds
\]
where $I$ ranges over all smooth curves $\gamma(\cdot)$ such that $\dot\gamma(s)\in D_{\gamma(s)}$ for all $t$ in $[s_0,s_1]$.

\begin{cor}\label{keycorSasakian1}
Assume that the manifold is compact Sasakian and it satisfies $\overline\ric\geq 0$. Let $U_1$ and $U_2$ be two smooth functions on $M$ satisfying
\begin{enumerate}
\item $V=\Delta_\hor U_1+\frac{1}{2}|\nabla_\hor U_1|^2-2U_2$
\item $V\leq \kappa_1$,
\item $\Delta_\hor V+\frac{n^2}{3\kappa_2}\left(1+\frac{3}{n}\right)^2|\nabla_\ver V|^2\leq \kappa_2$,
\end{enumerate}
for some positive constants $\kappa_1$ and $\kappa_2$. Then any positive solution of the equation
\[
\dot \rho_t=\Delta_\hor\rho_t+\left<\nabla_\hor U_1,\nabla_\hor\rho_t\right>+U_2\rho_t
\]
satisfies
\[
\begin{split}
\frac{\rho_{s_1}(x_1)}{\rho_{s_0}(x_0)}&\geq \left(\frac{\sinh(c_2s_1)}{\sinh(c_2s_0)}\right)^{-(n+3)}\\
&\cdot \exp\Big(-\frac{1}{2}\left(U(x_1)-U(x_0)+\left(1+\frac{3}{n}\right)c_{s_0,s_1}(x_0,x_1)\right)\Big).
\end{split}
\]
for all $s_0, s_1\geq 0$ and all $x_0, x_1$ in $M$, where $c_2=\frac{1}{n+3}\sqrt{\frac{n\kappa_2}{2}}$.
\end{cor}

The structure of the paper is as follows. The main results of this paper are stated in Section 2. In \cite{Le1}, a moving frame argument was used for the proof of Theorem \ref{Lee1} instead of the Bochner formula. The advantage is that a matrix version of Theorem \ref{Lee1}, generalizing the matrix Hamilton-Li-Yau estimate for the heat equation \cite{Ha1}, can be proved using a very similar argument. Although there is no matrix analogue of Theorem \ref{keythmSasakian} in this paper, we show that a version of the moving frame argument is possible in the present setting. This is done in Section 3 and 4. The proofs of the main results are given in Section 5. Section 6 is an appendix devoted to some calculations needed in the proofs.

\smallskip

\section{The main results}

In this section, we give the statements of the main results. First, let us introduce the setup which is essentially the same as that of \cite{BaGa}. 

Let $M$ be a Riemannian manifold and let us fix a distribution $D$ (a vector bundle of the tangent bundle $TM$) of rank $k$. We assume that the orthogonal complement of $D$ is spanned by $n-k$ vector fields denoted by $\Rb_1,...,\Rb_{n-k}$ which satisfy certain symmetry conditions to be specified. We will also assume that the vector field $X_t$ is the horizontal gradient $\nabla_\hor f_t$ of a one-parameter family of functions $f_t$ defined on the manifold $M$ and specialize Lemma \ref{Ncomp} to this case.

Let us call vectors or vector fields which are contained in the distribution $D$ horizontal. Let $\psi_t$ be the flow of a vector field $\Rb$. Assume that $\psi_t$ sends horizontal vector fields to horizontal ones and  preserves their lengths. If $X_1$ and $X_2$ are horizontal vector fields, then we have
\[
\left<(\psi_t)_*(X_1),\mathbf {v}\right>=0\text{ and } \left<(\psi_t)_*(X_1),(\psi_t)_*(X_2)\right>=\left<X_1,X_2\right>.
\]
for any vector field $\mathbf {v}$ which is in the orthogonal complement of $D$.

If we differentiate the above  equations with respect to $t$, then we obtain
\[
\left<[\Rb, X_1],\tilde\Rb\right>=0\text{ and } \left<\nabla_{X_1}\Rb,X_2\right>+\left<X_1, \nabla_{X_2}\Rb\right> =0.
\]

Therefore, we call a vector field $\Rb$ which satisfies the following two conditions horizontal isometry:
\begin{itemize}
\item $[\Rb_i, X_1]$ is horizontal,
\item $\left<\nabla_{X_1}\Rb,X_2\right>+\left<X_1, \nabla_{X_2}\Rb\right> =0$,
\end{itemize}
for all horizontal vector fields $X_1$ and $X_2$.

Let $v$ be a tangent vector of $M$. Then the projections of $v$ onto the distribution $D$ and its orthogonal complement $D^\perp$ are called the horizontal part $v_\hor$ and the vertical part $v_\ver$ of $v$, respectively. Let $f:M\to\Real$ be a smooth function. For the notation convenience, we also denote the horizontal part and vertical part of the gradient $\nabla f$ by $\nabla_\hor f$ and $\nabla_\ver f$, respectively. Let $v_1,...,v_k$ be a frame in $D$ which is orthonormal. Then the sub-Laplacian of $f$ is defined by
\[
\Delta_\hor f=\sum_i\left<\nabla_{v_i}\nabla f,v_i\right>.
\]

Recall that the Ricci curvature $\ric(v,v)$ is defined as the trace of the following operator $w\mapsto \left<\Rm(w,v)v,w\right>$. We define the horizontal Ricci curvature $\ric^\hor$ by
\[
\ric^\hor(v,v)=\sum_i\left<\Rm(v_i,v)v,v_i\right>
\]
and the vertical Ricci curvature $\ric^\ver$ by
\[
\ric^\ver(v,v)=\sum_i\left<\Rm(u_i,v)v,u_i\right>.
\]

Let $\rho_t$ be a smooth positive solution of the following equation
\[
\dot\rho_t=\Delta_\hor\rho_t+\left<\nabla U_1,\nabla_\hor\rho_t\right> +U_2\rho_t+K\rho_t\log\rho_t
\]
where $U_1, U_2$ are smooth functions on $M$ and $K$ is a constant.

Let $f_t$ be the one-parameter family of smooth functions defined by 
\[
f_t=-2\log\rho_t-U_1.
\]
A computation shows that $f_t$ satisfies the following equation
\begin{equation}\label{eqn}
\dot f_t+\frac{1}{2}|\nabla_\hor f_t|^2=\Delta_\hor f_t+ V + Kf_t,
\end{equation}
where $V=\Delta_\hor U_1+KU_1+\frac{1}{2}|\nabla_\hor U_1|^2-2U_2$.

We call a solution $r$ of the problem
\[
\dot r(t)=F(r(t)),\quad r(t)\to\infty \text{ as } t\to 0^+
\]
stable if there is a family of solutions $r_\epsilon$ of the following
\[
\dot r_\epsilon(t)=F(r(t))+\epsilon,\quad r_\epsilon(t)\to\infty \text{ as } t\to 0^+
\]
such that $r_\epsilon$ converges pointwise to $r_\epsilon$.

Finally, recall that a distribution is involutive if the Lie bracket of any two sections in the distribution is again in the distribution. The following is the main result of this paper.

\begin{thm}\label{keythm}
Assume that the orthogonal complement $D^\perp$ of the distribution $D$ is involutive and is given by the span of $n-k$ horizontal isometries. Assume also that the following conditions hold: 
\begin{enumerate}
\item $\ric^\hor(v,v)+3\ric^\ver(v_\hor,v_\hor)\geq K_1|v_\hor|^2 +K_2|v_\ver|^2$,
\item $\ric^\ver(v_\hor,v_\hor)\leq  K_3|v_\hor|^2$,
\item $V\leq K_4$,
\item $\Delta_\hor V+K_5|\nabla_\ver V|^2-2K_1V\leq K_6$,
\item $a_3(t)+\frac{4K_3}{a_2(t)}\geq 0$,
\item $\dot a_1(t)+\frac{2a_1(t)a_4(t)}{k}+(a_1(t)+1)\left(2K_1-a_3(t) -\frac{8K_3}{a_2(t)}\right)=0$,
\item $\frac{\dot a_2(t)}{2}-K_2-4K_3+a_2(t)\left(\frac{a_2(t)}{4K_5}-\frac{a_3(t)}{2}+\frac{a_4(t)}{k} - \frac{K}{2}+K_1\right)=0$,
\item $\dot a_3(t)+\frac{2a_3(t)a_4(t)}{k}+(K-a_3(t))\left(a_3(t)+\frac{8K_3}{a_2(t)}-2K_1\right)=0$.
\end{enumerate}
for some positive constants $K_1, K_2, K_3, K_4, K_5, K_6$.

Let $r(\cdot)$ be a stable solution of
\[
\begin{split}
\dot r(t) &= \frac{a_4(t)^2}{k} +K_6+2\left(a_3(t)+\frac{4K_3}{a_2(t)}\right)K_4\\
&+\left(a_3(t)-\frac{2a_4(t)}{k} + K+\frac{8K_3}{a_2(t)}-2K_1\right)r(t)
\end{split}
\]
with the condition $r(t)\to \infty$ as $t\to 0^+$.

Then
\[
\Delta_\hor f_t(x)+a_1(t)\dot f_t(x)+\frac{a_2(t)}{2}|\nabla_\ver f_t(x)|^2+a_3(t)f_t(x)\leq r(t)
\]
for all $t\geq 0$ and all $x$ in $M$.
\end{thm}

In the case $\dot\rho_t=\Delta_\hor\rho_t$, the proof of Theorem \ref{keythm} gives the following result which is one of the results in \cite{BaGa}. Note that, unlike \cite{BaGa}, the curvature conditions of the following result is written using a completely Riemannian notations.

\begin{cor}\label{keycor1}
Assume that the orthogonal complement $D^\perp$ of the distribution $D$ is involutive and is given by the span of $n-k$ horizontal isometries. Assume also that the following conditions hold:

\begin{enumerate}
\item $\ric^\hor(v,v)+3\ric^\ver(v_\hor,v_\hor)\geq K_2|v_\ver|^2$,
\item $\ric^\ver(v_\hor,v_\hor)\leq  K_3|v_\hor|^2$,
\end{enumerate}
for some positive constants $K_2, K_3$. Let $\rho_t$ be a smooth positive solution of the equation 
\[
\dot \rho_t=\Delta_\hor\rho_t
\]
and let $f_t=-2\log\rho_t$. Then
\[
\begin{split}
\Delta_\hor f_t(x)+c\dot f_t(x)+&2t\left(K_2-\frac{4K_3}{c}\right)|\nabla_\ver f_t(x)|^2\\
&\leq\frac{4(c+1)^2kK_3^2}{(cK_2-4K_3)(8K_3-cK_2)t}
\end{split}
\]
for all $c>\frac{4K_3}{K_2}$, all $t\geq 0$, and all $x$ in $M$.
\end{cor}

If we set $c=\frac{6K_3}{K_2}$ in Corollary \ref{keycor1}, then the result further simplified to the following. 

\begin{cor}\label{keycor2}
Suppose that the assumptions in Corollary \ref{keycor1} hold. Then
\[
\begin{split}
\Delta_\hor f_t(x)+\frac{6K_3}{K_2}\dot f_t(x)+&\frac{2K_2t}{3}|\nabla_\ver f_t(x)|^2\leq\frac{k\left(1+\frac{6K_3}{K_2}\right)^2}{t}
\end{split}
\]
for all $t\geq 0$ and all $x$ in $M$.
\end{cor}

In the case $\dot\rho_t=\Delta_\hor\rho_t+\left<\nabla U_1,\nabla_\hor\rho_t\right> +U_2\rho_t$, Theorem \ref{keythm} gives

\begin{cor}\label{keycor3}
Assume that the orthogonal complement $D^\perp$ of the distribution $D$ is involutive and is given by the span of $n-k$ horizontal isometries. Assume also that the following conditions hold:
\begin{enumerate}
\item $\ric^\hor(v,v)+3\ric^\ver(v_\hor,v_\hor)\geq K_2|v_\ver|^2$,
\item $\ric^\ver(v_\hor,v_\hor)\leq  K_3|v_\hor|^2$,
\item $V\leq K_4$,
\item $\Delta_\hor V+K_5|\nabla_\ver V|^2\leq K_6$,
\item $K_5=\frac{4k(c+1)^2K_3^2}{cK_6(8K_3-cK_2)}$,
\end{enumerate}
for some positive constants $K_2, K_3, K_4, K_6$ and $c>\frac{4K_3}{K_2}$. Let $\rho_t$ be a smooth positive solution of the equation
\[
\dot \rho_t=\Delta_\hor\rho_t+\left<\nabla U_1,\nabla_\hor\rho_t\right> +U_2\rho_t,
\]
$f_t=-2\log\rho_t-U_1$, $c_1=2\sqrt{K_5\left(K_2-\frac{4K_3}{c}\right)}$, and $c_2=\sqrt{\frac{cK_2-4K_3}{cK_5}}$. Then
\[
\Delta_\hor f_t(x)+c\dot f_t(x)+\frac{c_1\tanh(c_2t)}{2}|\nabla_\ver f_t(x)|^2\leq \frac{K_6}{c_2}\coth(c_2t)+cK_4
\]
for all $t\geq 0$ and all $x$ in $M$.
\end{cor}

Next, we recall the definition of Sasakian manifolds and show that Theorem \ref{keythmSasakian} is a consequence of Corollary \ref{keycor3}. For a more detail discussion of Sasakian manifolds, see \cite{Bl}.

Let $M$ be a $2n+1$ dimensional manifold. A 1-form $\alpha$ on $M$ is contact if $d\alpha_x$ is a non-degenerate 2-form on the kernel $D$ of $\alpha$ (i.e.
\[
D_x=\{v\in T_xM|\alpha(v)=0\}
\]
for each $x$).

Let $J$ be a $(1,1)$-tensor, $\Rb$ be a vector field, and $\alpha$ be a contact 1-form on $M$. The triple $(J,\Rb,\alpha)$ is an almost contact structure of $M$ if the following conditions hold
\begin{equation}\label{almostcontact}
J^2(v)=-v,\quad \alpha(\Rb)=1,\quad  J(\Rb)=0,
\end{equation}
where $v$ is any vector in $D$.

An almost contact structure $(J,\Rb,\alpha)$ is normal if
\begin{equation}\label{normal}
[J,J](w_1,w_2)+d\alpha(w_1,w_2)\Rb=0
\end{equation}
for any vector fields $w_1$ and $w_2$ on $M$, where $[J,J]$ denotes the Nijenhuis tensor defined by
\begin{equation}\label{Nijen}
[J,J](w_1,w_2)=J^2[w_1,w_2]+[Jw_1,Jw_2]-J[Jw_1,w_2]-J[w_1,Jw_2].
\end{equation}

An almost contact structure $(J,\Rb,\alpha)$ together with a Riemannian metric $\left<\cdot,\cdot\right>$ is called a almost contact metric structure if
\[
\left<Jv_1,Jv_2\right>=\left<v_1,v_2\right>\quad\text{ and } \quad \left<\Rb,\Rb\right>=1,
\]
where $v_1$ and $v_2$ are vectors in $D$.

An almost contact metric structure is a contact metric structure if $\left<w_1,Jw_2\right>=d\alpha(w_1,w_2)$. A Sasakian manifold is a manifold $M$ equipped with a contact metric structure which is normal.

An example of Sasakian manifolds is given by the Heisenberg group. The underlying space $M$ of the Heisenberg group is the $2n+1$ dimensional Euclidean space $\Real^{2n+1}$. In this case, the contact form $\alpha$ is given by
\[
\alpha=dz-\frac{1}{2}\sum_{i=1}^n\left(y_idx_i-x_idy_i\right)
\]
where $\{x_1,...,x_n,y_1,...,y_n,z\}$ are coordinates on $\Real^{2n+1}$.

In the Heisenberg group, the vector field $\Rb$ is given by $\Rb=\partial_z$. Let $X_i=\partial_{x_i}+\frac{1}{2}y_i\partial_z$ and $Y_i=\partial_{y_i}-\frac{1}{2}x_i\partial_z$. The Riemannian metric is defined such that $\{X_1,...,X_n,Y_1,...,Y_n,\Rb\}$ is an orthonormal frame. The tensor $J$ is defined by
\[
JX_i=-Y_i,\quad JY_i=X_i,\quad\text{ and }\quad J\Rb =0.
\]

Back to the general case, the Riemann curvature tensor $\Rm$ of a Sasakian manifold satisfies the following properties.

\begin{prop}\label{SasakianProp}
Let $(M, J, \Rb,\alpha,\left<\cdot,\cdot\right>)$ be a Sasakian manifold. Then the followings hold:
\begin{enumerate}
\item $(\nabla_{w_1} J)w_2=\frac{1}{2}\left(\left<w_1,w_2\right>\Rb-\left<\Rb,w_2\right>w_1\right)$,
\item $\nabla_{w_1}\Rb=-\frac{1}{2}Jw_1$,
\item $\Rb$ is a Killing vector field,
\item $\Rm(w_1,w_2)\Rb=\frac{1}{4}\left(\left<\Rb,w_2\right>w_1-\left<\Rb,w_1\right>w_2\right)$,
\end{enumerate}
for all tangent vectors $w_1$ and $w_2$ on $M$.
\end{prop}

A proof of the above proposition can be found in \cite{Bl}. Note that the definition of the exterior differential $d\alpha$ of a differential 2-form $\alpha$ used in this paper is
\[
d\alpha(w_1,w_2)=w_1(\alpha(w_2))-w_2(\alpha(w_1))-\alpha([w_1,w_2]).
\]
This is different from the one in \cite{Bl} and so the above formulas are also different from those in \cite{Bl} by a multiplicative constant.

Recall that the Tanaka connection $\bar\nabla$ of a given almost contact metric manifold is given by
\[
\begin{split}
\bar\nabla_{w_1}w_2&=\nabla_{w_1}w_2+\frac{1}{2}\left<\Rb,w_1\right>Jw_2 +\frac{1}{2}\left<\Rb,w_2\right>Jw_1-\frac{1}{2}\left<Jw_1,w_2\right>\Rb.
\end{split}
\]
The corresponding curvature $\overline\Rm$ is given by
\[
\overline\Rm(w_1,w_2)w_3=\bar\nabla_{w_1}\bar\nabla_{w_2}w_3-\bar\nabla_{w_2}\bar\nabla_{w_1}w_3-\bar\nabla_{[w_1,w_2]}w_3
\]
and we denote by $\overline\ric$ the corresponding Ricci curvature
\[
\overline\ric(w_1,w_1)=\text{trace}(w_2\mapsto\left<\overline\Rm(w_2,w_1)w_1,w_2\right>).
\]

The following proposition shows that Sasakian manifolds provide examples to the main results.

\begin{prop}\label{Sasak}
The followings hold on a Sasakian manifold:
\begin{enumerate}
\item $\ric^\ver(v,v)=\ric^\ver(v_\hor,v_\hor)=\frac{1}{4}|v_\hor|^2$,
\item $\ric^\hor(v,v)+3\ric^\ver(v_\hor,v_\hor)=\ric^\hor(v,v)+\frac{3}{4}|v_\hor|^2\\
=\frac{n}{2}|v_\ver|^2+\overline\ric(v_\hor, v_\hor)$.
\end{enumerate}
\end{prop}

\begin{proof}
Clearly, we have
\[
\begin{split}
\ric^\ver(v,v)&=\left<\Rm(\Rb,v)v,\Rb\right>\\
&=\left<\Rm(\Rb,v_\hor)v_\hor,\Rb\right>=\ric^\ver(v_\hor,v_\hor).
\end{split}
\]

By Proposition \ref{SasakianProp}, we also have
\[
\begin{split}
\left<\Rm(\Rb,v_\hor)v_\hor,\Rb\right>&=\left<\Rm(v_\hor,\Rb)\Rb,v_\hor\right>=\frac{1}{4}|v_\hor|^2.
\end{split}
\]

By Proposition \ref{SasakianProp} again, we have
\begin{equation}\label{R0}
\begin{split}
&\left<\Rm(w_1,v)v,w_2\right>=\left<\Rm(w_1,v_\hor)v_\hor,w_2\right>\\
&+\left<\Rb,v\right>\left<\Rm(w_1,v)\Rb,w_2\right> +\left<\Rb,v\right>\left<\Rm(w_1,\Rb)v_\hor,w_2\right>\\
&=\left<\Rm(w_1,v_\hor)v_\hor,w_2\right>+\frac{1}{4}\left<\Rb,v\right>^2\left<w_1,w_2\right>\\
&=\left<\Rm(w_1,v_\hor)v_\hor,w_2\right>+\frac{|v_\ver|^2}{4}\left<w_1,w_2\right>.
\end{split}
\end{equation}

By the definition of Tanaka connection,
\[
\bar\nabla_XY
\]
is horizontal for any vector fields $X$ and $Y$. Therefore, by Proposition \ref{SasakianProp},
\begin{equation}\label{R1}
\begin{split}
&\left<\bar\nabla_{v_i}\bar\nabla_{v_\hor}v_\hor,v_j\right>=\left<\nabla_{v_i}\bar\nabla_{v_\hor}v_\hor,v_j\right>\\
&=\left<\nabla_{v_i}(\nabla_{v_\hor}v_\hor-\left<\nabla_{v_\hor}v_\hor,\Rb\right>\Rb),v_j\right>\\
&=\left<\nabla_{v_i}\nabla_{v_\hor}v_\hor,v_j\right>+\frac{1}{2}\left<\nabla_{v_\hor}v_\hor,\Rb\right>\left<Jv_i,v_j\right>\\
&=\left<\nabla_{v_i}\nabla_{v_\hor}v_\hor,v_j\right>.
\end{split}
\end{equation}
Here we extend $v$ to a vector field and still call it $v$.

Similarly, we also have
\begin{equation}\label{R2}
\begin{split}
&\left<\bar\nabla_{v_\hor}\bar\nabla_{v_i}v_\hor,v_j\right>=\left<\nabla_{v_\hor}\bar\nabla_{v_i}v_\hor,v_j\right>\\
&=\left<\nabla_{v_\hor}(\nabla_{v_i}v_\hor-\left<\nabla_{v_i}v_\hor,\Rb\right>\Rb),v_j\right>\\
&=\left<\nabla_{v_\hor}\nabla_{v_i}v_\hor,v_j\right>+\frac{1}{2}\left<\nabla_{v_i}v_\hor,\Rb\right>\left<Jv,v_j\right>\\
&=\left<\nabla_{v_\hor}\nabla_{v_i}v_\hor,v_j\right>-\frac{1}{4}\left<Jv,v_i\right>\left<Jv,v_j\right>.
\end{split}
\end{equation}

By the definition of Tanaka connection, it also follows that
\begin{equation}\label{R3}
\begin{split}
&\left<\bar\nabla_{[v_i,v_\hor]}v_\hor,v_j\right>\\
&=\left<\nabla_{[v_i,v_\hor]}v_\hor,v_j\right>+\frac{1}{2}\left<\Rb,[v_i,v_\hor]\right>\left<Jv,v_j\right>.
\end{split}
\end{equation}

Therefore, by combining (\ref{R1}), (\ref{R2}), and (\ref{R3}), we obtain
\begin{equation}\label{R4}
\begin{split}
&\left<\overline\Rm(v_i,v_\hor)v_\hor,v_j\right>=\left<\Rm(v_i,v_\hor)v_\hor,v_j\right>\\
&\quad +\frac{1}{4}\left<Jv,v_i\right>\left<Jv,v_j\right> - \frac{1}{2}\left<\Rb,[v_i,v_\hor]\right>\left<Jv,v_j\right>\\
&=\left<\Rm(v_i,v_\hor)v_\hor,v_j\right> +\frac{3}{4}\left<Jv,v_i\right>\left<Jv,v_j\right>.
\end{split}
\end{equation}
Since $\bar\nabla_XY$ is horizontal, we also have
\begin{equation}\label{R5}
\left<\overline\Rm(w,v_\hor)v_\hor,\Rb\right>=\left<\overline\Rm(\Rb,v_\hor)v_\hor,w\right>=0
\end{equation}
for any vector field $w$.

Therefore, by (\ref{R0}), (\ref{R4}), and (\ref{R5}), the second assertion follows.
\end{proof}

\begin{proof}[Proof of Theorem \ref{keythmSasakian}]
It follows immediately from Corollary \ref{keycor2} and Proposition \ref{Sasak} with $c=\frac{6K_3}{K_2}$, $K_2=\frac{n}{2}$, $K_3=\frac{1}{4}$, and $k=2n$.
\end{proof}

\begin{proof}[Proof of Corollary \ref{keycorSasakian1}]
Let $\gamma(\cdot)$ be a minimizer of the functional
\[
\gamma(\cdot)\mapsto \inf_I\int_{s_0}^{s_1}\frac{1}{2}|\dot\gamma(s)|^2+W(\gamma(s))ds,
\]
where $I$ ranges over all smooth curves $\gamma(\cdot)$ such that $\dot\gamma(s)\in D_{\gamma(s)}$ for all $t$ in $[s_0,s_1]$.

By Theorem \ref{keythmSasakian},
\[
\begin{split}
&\left(1+\frac{3}{n}\right)\frac{d}{dt}f_t(\gamma(t))=\left(1+\frac{3}{n}\right)\dot f_t(\gamma(t))+\left(1+\frac{3}{n}\right)\left<\nabla_\hor f_t(\gamma(t)),\dot \gamma(t)\right>\\
&\leq -\frac{1}{2}|\nabla_\hor f_t|^2_x+V(x)+\frac{K_6}{c_2}\coth(c_2t)+\frac{3K_4}{n}+\left(1+\frac{3}{n}\right)\left<\nabla_\hor f_t(\gamma(t)),\dot \gamma(t)\right>.
\end{split}
\]

By Young's inequality, we have
\[
\begin{split}
\left(1+\frac{3}{n}\right)\frac{d}{dt}f_t(\gamma(t))&\leq V(\gamma(t))+\frac{K_6}{c_2}\coth(c_2t)+\frac{3K_4}{n}+\frac{1}{2}\left(1+\frac{3}{n}\right)^2|\dot \gamma(t)|^2.
\end{split}
\]

By integrating the above inequality, we obtain
\[
\begin{split}
\left(1+\frac{3}{n}\right)\left(f_{s_1}(x_1)-f_{s_0}(x_0)\right)&\leq\frac{K_6}{c_2^2}\ln\left(\frac{\sinh(c_2s_1)}{\sinh(c_2s_0)}\right) +\left(1+\frac{3}{n}\right)^2c_{s_0,s_1}(x_0,x_1).
\end{split}
\]

This result follows from this.
\end{proof}

\smallskip

\section{Parallel adapted frames}

In this section, we define convenient adapted frames along a path called parallel adapted frames and use it to see how the linearization of a flow changes.

Let $M$ be a compact Riemannian manifold of dimension $n$ equipped with a distribution $D$ (i.e. a sub-bundle of the tangent bundle) of rank $k$. An orthonormal frame $v_1,...,v_k,u_1,...,u_{n-k}$ in the tangent space $T_xM$ at a point $x$ is an adapted frame if $v_1,...,v_k$ is contained in the space $D_x$.

\begin{lem}\label{parallel}
Let $\gamma:[0,T]\to M$ be a smooth path in $M$. Then there exists a 1-parameter family of adapted frames
\[
\{v_1(t),...,v_k(t),u_1(t),...,u_{n-k}(t)\}
\]
along $\gamma(t)$  such that
\begin{enumerate}
\item $\dot v_i(t)$ is in the orthogonal complement of $D_{\gamma(t)}$ for $1\leq i\leq k$,
\item $\dot u_j(t)$ is in $D_{\gamma(t)}$ for $1\leq j\leq n-k$.
\end{enumerate}
Here $\dot v$ denotes the covariant derivative of $v(\cdot)$ along $\gamma(\cdot)$.

Moreover, if $\{\bar u_1(t),...,\bar u_{n-k}(t),\bar v_1(t),...,\bar v_k(t)\}$ is another such frame, then there are orthogonal matrices $O^{(1)}$ and $O^{(2)}$ of sizes $k\times k$ and $(n-k)\times(n-k)$ (independent of time $t$), respectively, such that
\[
\begin{split}
&\bar v_i(t)=\sum_{l=1}^k O_{il}^{(1)}v_l(t),\quad 1\leq i\leq k,\\
&\bar u_j(t)=\sum_{s=1}^{N-k} O_{js}^{(2)}u_s(t),\quad 1\leq j\leq n-k.
\end{split}
\]
\end{lem}

Let $u_1,...,u_{n-k}, v_1,...,v_k$ be an adapted frame in $T_{\gamma(0)}M$. According to the above lemma, there is a unique 1-parameter family of adapted frames $u_1(t),...,u_{n-k}(t), v_1(t),...,v_k(t)$ defined along $\gamma$ such that $v_i(0)=v_1$ and $u_j(0)=u_j$ for each $i=1,...,k$ and each $j=1,...,n-k$. We call $u_1(t),...,u_{n-k}(t), v_1(t),...,v_k(t)$ a parallel adapted frame defined along $\gamma$ which starts from $u_1,...,u_{n-k}, v_1,...,v_k$.

\begin{rem}
The notion of parallel adapted frame is a generalization of parallel transported frame in Riemannian geometry to the present setting (see the end of this section for detail). The construction of the above frame can be generalized to a more complicated setting where we have a smoothly varying flags of subspaces instead of just one distribution.
\end{rem}

Let $\varphi_t$ be the flow of a vector field $X_t$ defined by $\dot\varphi_t=X_t(\varphi_t)$ and $\varphi_0(x)=x$. The linearization $d\varphi_t$ of $\varphi_t$ satisfies
\[
\frac{d}{dt} d\varphi_t(w)=\nabla_{d\varphi_t(w)} X_t.
\]
Here, we use $\frac{d}{dt}$ to denote covariant derivative along $t\mapsto \varphi_t(x)$.

Therefore, the change in $d\varphi_t$ is completely determined by the $(1,1)$-tensor $w\mapsto \nabla_w X_t$. We will investigate the equation satisfied by this tensor. Let $N(t)$ be the matrix representation of the $(1,1)$-tensor $w\mapsto \nabla_w X_t$ with respect to the above parallel adapted frame at time $t$. More precisely, the $ij$-th entry $N_{ij}(t)$ of $N(t)$ is defined by
\[
N_{ij}(t)=\begin{cases}
  \left<\nabla_{u_i(t)} X_t,u_j(t)\right> & \mbox{ if $i\leq n-k$ and $j\leq n-k$}\\
  \left<\nabla_{u_i(t)} X_t,v_{j-n+k}(t)\right> & \mbox{ if $i\leq n-k$ and $j> n-k$}\\
  \left<\nabla_{v_{i-n+k}(t)} X_t,u_j(t)\right> & \mbox{ if $i> n-k$ and $j\leq n-k$}\\
 \left<\nabla_{v_{i-n+k}(t)} X_t,v_{j-n+k}(t)\right> & \mbox{ if $i> n-k$ and $j> n-k$}.
\end{cases}
\]

Similarly, let $R(t)$ and $M(t)$ be the matrix representations of the bilinear form
\[
w\mapsto\left<\Rm(w,X_t)X_t,w\right>
\]
and the $(1,1)$-tensor
\[
w\mapsto \nabla_{w}(\dot X_t+\nabla_{X_t}X_t),
\]
respectively, with respect to the given parallel adapted frame at time $t$. Finally, let
\[
\mathfrak W(t)=\left(
            \begin{array}{cc}
              0_{(n-k)\times(n-k)} & W(t) \\
              -W(t)^T & 0_{k\times k}
            \end{array}
          \right),
\]
where $W(t)$ is the $(n-k)\times k$ matrix with $ij$-th entry equal to
\[
W_{ij}(t)=\left<\dot u_i(t),v_j(t)\right>.
\]

\begin{lem}\label{N}
The 1-parameter family of matrices $N(t)$ satisfies the following matrix Riccati equation
\[
\begin{split}
\dot N(t)=-N(t)^2-N(t)\mathfrak W(t)-\mathfrak W(t)^TN(t)-R(t)+M(t).
\end{split}
\]
\end{lem}

Finally, we split each of $N(t)$, $R(t)$, and $M(t)$ into four pieces
\[
\begin{split}
&N(t)=\left(
       \begin{array}{cc}
         \mathcal N_{00}(t) & \mathcal N_{01}(t) \\
         \mathcal N_{10}(t) & \mathcal N_{11}(t) \\
       \end{array}
     \right), \quad R(t)=\left(
       \begin{array}{cc}
         \mathcal R_{00}(t) & \mathcal R_{01}(t) \\
         \mathcal R_{10}(t) & \mathcal R_{11}(t) \\
       \end{array}
     \right),\\
      &\quad M(t)=\left(
       \begin{array}{cc}
         \mathcal M_{00}(t) & \mathcal M_{01}(t) \\
         \mathcal M_{10}(t) & \mathcal M_{11}(t) \\
       \end{array}
     \right),
\end{split}
\]
where $\mathcal N_{00}(t)$, $\mathcal R_{00}(t)$, and $\mathcal M_{00}(t)$ are of size $(n-k)\times (n-k)$.

The following, which will be used in the later sections, is an immediate consequence of Lemma \ref{N}.

\begin{lem}\label{Ncomp}
The 1-parameter family of matrices $N(t)$ satisfies the following
\[
\begin{split}
\dot{\mathcal N}_{11}(t)&=-\mathcal N_{11}(t)^2-\mathcal N_{10}(t)\mathcal N_{01}(t)\\
&\quad -\mathcal N_{10}(t)W(t)-W(t)^T\mathcal N_{01}(t)-\mathcal R_{11}(t)+\mathcal M_{11}(t).
\end{split}
\]
\end{lem}

\begin{proof}[Proof of Lemma \ref{parallel}]
Let $(\bar v_1(t),...,\bar v_k(t),\bar u_1(t),...,\bar u_{n-k}(t))$ be a 1-parameter family of adapted frames along the path $\gamma(t)$. Let
\[
(v_1(t),...,v_k(t), u_1(t),...,u_{n-k}(t))
\]
be any other such family. Then
\[
\begin{split}
&v_i(t)=\sum_{l=1}^kO_{il}(t)\bar v_l(t),\quad 1\leq i\leq k,\\
\end{split}
\]
where $O(t)$ are orthogonal matrices of size $k\times k$.

Let $A(t)$ be the $k\times k$ matrix with $ij$-th entry equal to $\left<\dot{\bar v}_i(t),\bar v_j(t)\right>$. Then $\left<\dot v_i(t),v_j(t)\right>=0$ for each $1\leq i, j\leq k$ if and only if
\[
\dot O(t)=-O(t)A(t).
\]
Since $A(t)$ is skew symmetric, we have a solution $O(t)$ to the above ODE. Moreover, any solution is determined by its initial condition. This proves the result for $v_i(\cdot)$. A similar procedure gives the result for $u_i(\cdot)$.
\end{proof}

\begin{proof}[Proof of Lemma \ref{N}]
Let
\[
V(t)=(u_1(t),...,u_{n-k}(t),v_1(t),...,v_k(t))^T
\]
and
\[
\dot V(t)=(\dot u_1(t),...,\dot u_{n-k}(t),\dot v_1(t),...,\dot v_k(t))^T,
\]
where $\{u_1(t),...,u_{n-k}(t),v_1(t),...,v_k(t)\}$ is a parallel adapted frame defined along the curve $t\mapsto\varphi_t(x)$ as in Lemma \ref{parallel}.

By Lemma \ref{parallel}, we have
\[
\dot V(t)=\left(
            \begin{array}{cc}
              0_{(N-k)\times (N-k)} & W(t) \\
              -W(t)^T & 0_{k\times k}
            \end{array}
          \right)V(t)=\mathfrak W(t)V(t).
\]

If we differentiate the above equation once more, then we obtain
\[
\ddot V(t)=\dot {\mathfrak W}(t)V(t)+\mathfrak W(t)\dot V(t)=(\dot {\mathfrak W}(t)+\mathfrak W(t)^2) V(t).
\]

Let $\Phi_t=(d\varphi_t(v_0(0)),...,d\varphi_t(v_{2n}(0)))^T$ and let $A(t)$ be the matrices defined by
\[
\Phi_t=A(t)V(t).
\]
It follows that
\begin{equation}\label{E1}
\Lev t\Phi_t=(\dot A(t)+A(t)\mathfrak W(t))V(t)
\end{equation}
and
\begin{equation}\label{E2}
\begin{split}
\Lev t\Lev t\Phi_t&=\Big(\ddot A(t)+2\dot A(t)\mathfrak W(t)+A(t)\dot{\mathfrak W}(t)+A(t)\mathfrak W(t)^2\Big)V(t).
\end{split}
\end{equation}

On the other hand, if we let $\gamma(s)$ be a path such that $\gamma'(0)=v_i(0)$, then
\begin{equation}\label{E3}
\Lev td\varphi_t(v_i(0))=\Lev sX_t(\varphi_t(\gamma(s)))\Big|_{s=0}=\sum_{j=0}^{2n}A_{ij}(t)\nabla_{v_j(t)}X_t(\varphi_t(x)).
\end{equation}
By the definition of $\Rm$, we also have
\begin{equation}\label{E4}
\begin{split}
&\Lev t\Lev td\varphi_t(v_i(0))+\sum_{j=0}^{2n}A_{ij}(t)\Rm(v_j(t),X_t(\varphi_t(x)))X_t(\varphi_t(x))\\
&=\Lev s\Lev tX_t(\varphi_t(\gamma(s)))\Big|_{s=0}\\
&=\sum_{j=0}^{2n}A_{ij}(t)\left(\nabla_{v_j(t)}(\dot X_t+\nabla_{X_t}X_t)\right)(\varphi_t(x)).
\end{split}
\end{equation}

By (\ref{E1}) and (\ref{E3}), we have
\begin{equation}\label{E5}
N(t)=A(t)^{-1}\dot A(t)+\mathfrak W(t).
\end{equation}

By (\ref{E2}) and (\ref{E4}), we also have
\begin{equation}\label{E6}
\begin{split}
&A(t)^{-1}\ddot A(t)+2A(t)^{-1}\dot A(t)\mathfrak W(t)\\
&\quad +\dot{\mathfrak W}(t) +\mathfrak W(t)^2+R(t)-M(t)=0.
\end{split}
\end{equation}

If we combine (\ref{E5}) and (\ref{E6}), then we obtain
\[
\begin{split}
\dot N(t)=-N(t)^2-N(t)\mathfrak W(t)-\mathfrak W(t)^TN(t)-R(t)+M(t)
\end{split}
\]
as claimed.
\end{proof}

Before ending this section, let us discuss the relationships between parallel transported frames and the Tanaka connection. In the usual Riemannian case, if $v_1(t),...,v_n(t)$ is a parallel orthonormal frame defined along a path $\gamma$, then one can define the covariant derivative $\dot v(t)$ of a vector field
\[
v(t)=a_1(t)v_1(t)+...+a_n(t)v_n(t)
\]
defined along $\gamma(t)$ by
\[
\dot v(t)=\dot a_1(t)v_1(t)+...+\dot a_n(t)v_n(t).
\]
It is, of course, well-known that the covariant derivative is closely related to the corresponding Levi-Civita connection.

Similarly, one can define certain covariant derivative corresponding to the above parallel transported frames. More precisely, let
\[
\Rb(\gamma(t)), v_1(t),...,v_{2n}(t)
\]
be a parallel transported frame defined along a path $\gamma$ in an almost contact metric manifold. If
\[
v(t)=a_0(t)\Rb(\gamma(t))+a_1(t)v_1(t)+...+a_{2n}(t)v_{2n}(t)
\]
is a vector field defined along $\gamma(t)$, then the covariant derivative $\frac{\bar D}{dt}$ corresponding to the parallel transported frames of $v(t)$ along $\gamma$ is defined by
\[
\frac{\bar D}{dt} v(t)=\dot a_0(t)\Rb(\gamma(t))+\dot a_1(t)v_1(t)+...+\dot a_{2n}(t)v_{2n}(t).
\]
Note that the definition of $\frac{\bar D}{dt}$ is well-defined.

The following lemma gives some basic properties of $\frac{\bar D}{dt}$ and some of its relationships with the Tanaka connection $\bar\nabla$. Since it is not needed for the rest of the paper, the proof is omitted.

\begin{lem}\label{twcovar}
Let $w(t)$, $w_1(t)$, and $w_2(t)$ be three vector fields defined along a curve $\gamma$ in an almost contact metric manifold. Let $c_1$ and $c_2$ be two constants and let $a(t)$ be a smooth function. Then
\begin{enumerate}
\item $\Ta \left(c_1w_1(t)+c_2w_2(t)\right)=c_1\Ta w_1(t)+c_2\Ta w_2(t)$,
\item $\Ta a(t)w(t)=\dot a(t)w(t)+a(t)\Ta w(t)$,
\item $\Ta\Rb=0$,
\item if both $w_1(t)$ and $w_2(t)$ are contained in $D$, then
\[
\frac{d}{dt}\left<w_1(t),w_2(t)\right>=\left<\Ta w_1(t),w_2(t)\right>+\left<w_1(t),\Ta w_2(t)\right>,
\]
\item if $Y$ is a vector field contained in $D$, then
\[
\begin{split}
\Ta Y(\gamma(t))&=\bar \nabla_{\dot\gamma(t)}Y-\frac{1}{2}\left<\Rb,\dot\gamma(t)\right>JY.
\end{split}
\]
\end{enumerate}
\end{lem}

\smallskip

\section{Distributions with Transversal Symmetries}

In this section, we assume that the orthogonal complement of the given distribution $D$ is spanned by $n-k$ horizontal isometries denoted by $\Rb_1,...,\Rb_{n-k}$ We will also assume that the vector field $X_t$ in the previous section is the horizontal gradient $\nabla_\hor f_t$ of a one-parameter family of functions $f_t$ defined on the manifold $M$ and specialize Lemma \ref{Ncomp} to this case.

\begin{lem}\label{keylem}
Assume that the orthogonal complement $D^\perp$ of the distribution $D$ is involutive and is given by the span of $n-k$ horizontal isometries. If $\varphi_t$ is the flow of a time-dependent vector field $\nabla_\hor f_t$. Then
\[
\begin{split}
\frac{d}{dt}\Delta_\hor f_t(\varphi_t)&\leq -\frac{1}{k}(\Delta_\hor f_t(\varphi_t(x)))^2+ \Delta_\hor\left(\dot f_t+\frac{1}{2}|\nabla_\hor f|^2\right)\\
&-\ric^\hor_{\varphi_t(x)}(\nabla f_t,\nabla f_t)+\ric^\ver_{\varphi_t(x)}(\nabla_\hor f_t, \nabla_\hor f_t)\\
&-4\sum_k\left<\nabla_{\nabla_\hor f_t}\Rb_k,\nabla_{\Rb_k}\nabla_\hor f_t\right>_{\varphi_t(x)} .
\end{split}
\]
\end{lem}

\begin{proof}
Let us use the notation as in Lemma \ref{N} with $X_t=\nabla_\hor f_t$. Recall that $u_1(t), ..., u_{n-k}(t), v_1(t),...,v_k(t)$ is a parallel adapted frame along the path $t\mapsto\varphi_t(x)$, where $\varphi_t$ is the flow of $\nabla_\hor f_t$. By (1) of Lemma \ref{facts}, we have
\begin{equation}\label{N11}
\begin{split}
\tr(\mathcal N_{11}(t))&=\sum_i\left<\nabla_{v_i(t)}\nabla_\hor f_t,v_i(t)\right>\\
& = \sum_i\left<\nabla_{v_i(t)}\nabla f_t,v_i(t)\right> = \Delta_\hor f_t(\varphi_t(x)).
\end{split}
\end{equation}

By (1) of Lemma \ref{facts} and the symmetry of the Hessian, we also have
\[
\begin{split}
\tr(\mathcal N_{11}(t)^2)&=\sum_{i,j}\left<\nabla_{v_i(t)}\nabla_\hor f,v_j(t)\right>\left<\nabla_{v_j(t)}\nabla_\hor f,v_i(t)\right>\\
&= \sum_{i,j}\left(\left<\nabla_{v_i(t)}\nabla f,v_j(t)\right>-\left<\nabla_{v_i(t)}\nabla_\ver f,v_j(t)\right>\right)\\
&\quad \cdot\left(\left<\nabla_{v_j(t)}\nabla f,v_i(t)\right>-\left<\nabla_{v_j(t)}\nabla_\ver f,v_i(t)\right>\right)\\
&= \sum_{i,j}\left(\left<\nabla_{v_i(t)}\nabla f,v_j(t)\right>+\left<\nabla_{v_j(t)}\nabla_\ver f,v_i(t)\right>\right)\\
&\quad\cdot\left(\left<\nabla_{v_i(t)}\nabla f,v_j(t)\right>-\left<\nabla_{v_j(t)}\nabla_\ver f,v_i(t)\right>\right)\\
&= \sum_{i,j}\left<\nabla_{v_i(t)}\nabla f,v_j(t)\right>^2-\sum_{i,j}\left<\nabla_{v_j(t)}\nabla_\ver f,v_i(t)\right>^2.
\end{split}
\]

Therefore, by (11) of Lemma \ref{facts} and the Cauchy-Schwartz inequality, we have
\begin{equation}\label{N11sq}
\begin{split}
\tr(\mathcal N_{11}(t)^2)&\geq \frac{1}{k}(\Delta_\hor f_t(\varphi_t(x))^2-\ric^\hor_{\varphi_t(x)}(\nabla_\ver f,\nabla_\ver f).
\end{split}
\end{equation}

Let $O(t)$ be a family of orthogonal matrices such that
\[
u_i(t)=\sum_j O_{ij}(t)\Rb_j(\varphi_t(x)).
\]
It follows from (1) of Lemma \ref{facts} that
\[
\begin{split}
W_{ij}(t)&=\left<\dot u_i(t), v_j(t)\right>=\sum_{k}O_{ik}(t)\left<\nabla_{\nabla_\hor f_t}\Rb_k(\varphi_t(x)),v_j(t)\right>\\
&=-\sum_{k}O_{ik}(t)\left<\nabla_{v_j(t)}\Rb_k(\varphi_t(x)),\nabla_\hor f_t(\varphi_t(x))\right> \\
&=\sum_{k}O_{ik}(t)\left<\Rb_k(\varphi_t(x)),\nabla_{v_j(t)}\nabla_\hor f_t\right>=\left<\nabla_{v_j(t)}\nabla_\hor f_t,u_i(t)\right>.
\end{split}
\]

Therefore, we have
\begin{equation}\label{WTN10}
W(t)^T=\mathcal N_{10}(t).
\end{equation}

Therefore, the $ij$-th component of $W(t)^T\mathcal N_{01}(t)=\mathcal N_{10}(t)\mathcal N_{01}(t)$ is given by
\[
\begin{split}
&\sum_l\left<\dot u_l(t), v_i(t)\right>\left<\nabla_{u_l(t)}\nabla_\hor f_t,v_j(t)\right>\\
&=\sum_{k,l,s}O_{lk}(t)\left<\nabla_{\nabla_\hor f_t}\Rb_k(\varphi_t(x)),v_i(t)\right>O_{ls}(t)\left<\nabla_{\Rb_s}\nabla_\hor f_t(\varphi_t(x)),v_j(t)\right>\\
&=\sum_{k}\left<\nabla_{\nabla_\hor f_t}\Rb_k(\varphi_t(x)),v_i(t)\right>\left<\nabla_{\Rb_k}\nabla_\hor f_t(\varphi_t(x)),v_j(t)\right>
\end{split}
\]
and so
\begin{equation}\label{N10N01}
\begin{split}
\tr(\mathcal N_{10}(t)\mathcal N_{01}(t))&=\tr(W(t)^T\mathcal N_{01}(t))\\
&=\sum_{k}\left<(\nabla_{\nabla_\hor f_t}\Rb_k)_\hor,(\nabla_{\Rb_k}\nabla_\hor f_t)_\hor\right>_{\varphi_t(x)}.
\end{split}
\end{equation}

The $ij$-th component of $\mathcal N_{10}(t)\mathcal N_{10}(t)^T$ is
\[
\begin{split}
&\sum_{k, s} O_{lk}(t)\left<v_i(t),\nabla_{\nabla_\hor f}\Rb_k\right>O_{ls}(t)\left<\nabla_{\nabla_\hor f_t}\Rb_s,v_j(t)\right>\\
&=\sum_k \left<v_i(t),\nabla_{\nabla_\hor f}\Rb_k\right>\left<\nabla_{\nabla_\hor f_t}\Rb_k,v_j(t)\right>.
\end{split}
\]
Therefore,
\begin{equation}\label{N10}
\begin{split}
|\mathcal N_{10}(t)|^2&=\tr(\mathcal N_{10}(t)\mathcal N_{10}(t)^T)\\
&=\sum_k|(\nabla_{\nabla_\hor f_t}\Rb_k)_\hor|^2_{\varphi_t(x)}=\ric^\ver_{\varphi_t(x)}(\nabla_\hor f_t, \nabla_\hor f_t).
\end{split}
\end{equation}

By (1) and (3) of Lemma \ref{facts}, it follows that
\[
\begin{split}
&\tr(\mathcal M_{11}(t))=\sum_i\left<\nabla_{v_i(t)}(\nabla_\hor \dot f_t+\nabla_{\nabla_\hor f_t}\nabla_\hor f_t),v_i(t)\right>\\
&=\Delta_\hor\left(\dot f_t+\frac{1}{2}|\nabla_\hor f|^2\right)- 2\sum_i\left<\nabla_{v_i(t)}\nabla_{\nabla_\hor f}\nabla_\ver f,v_i(t)\right>.
\end{split}
\]

On the other hand, we have, by (8) and (12) of Lemma \ref{facts},
\[
\begin{split}
&\sum_i\left<\nabla_{v_i(t)}\nabla_{\nabla_\hor f_t}\nabla_\ver f_t,v_i(t)\right>=\sum_{i,l}\left<\nabla_{v_i(t)}\left(\Rb_lf\nabla_{\nabla_\hor f_t}\Rb_l\right),v_i(t)\right> \\
&=\ric^\hor_{\varphi_t(x)}(\nabla f_t,\nabla_\ver f_t)+\sum_l\left<\nabla_\hor(\Rb_l f_t),\nabla_{\nabla_\hor f_t}\Rb_l\right>_{\varphi_t(x)}\\
&=\ric^\hor_{\varphi_t(x)}(\nabla f_t,\nabla_\ver f_t)+\sum_l\left<(\nabla_{\Rb_l}\nabla_\hor f_t-\nabla_{\nabla_\hor f_t}\Rb_l)_\hor,\nabla_{\nabla_\hor f_t}\Rb_l\right>_{\varphi_t(x)}.
\end{split}
\]

Therefore, by combining this with (\ref{N11}), (\ref{N11sq}), (\ref{WTN10}), (\ref{N10N01}), and (\ref{N10}), we obtain
\[
\begin{split}
&\frac{d}{dt}\Delta_\hor f_t(\varphi_t)=\frac{d}{dt}\tr\left(\mathcal N_{11}(t)\right)\\
&=-\tr(\mathcal N_{11}(t)^2)-2\tr(\mathcal N_{10}(t)\mathcal N_{01}(t)) -|\mathcal N_{10}(t)|^2-\tr(\mathcal R_{11}(t))+\tr(\mathcal M_{11}(t))\\
&\leq -\frac{1}{k}(\Delta_\hor f_t(\varphi_t(x)))^2-\ric^\hor_{\varphi_t(x)}(\nabla f_t,\nabla f_t) +\ric^\ver_{\varphi_t(x)}(\nabla_\hor f_t, \nabla_\hor f_t)\\
& -4\sum_k\left<(\nabla_{\nabla_\hor f_t}\Rb_k)_\hor,(\nabla_{\Rb_k}\nabla_\hor f_t)_\hor\right>_{\varphi_t(x)} + \Delta_\hor\left(\dot f_t+\frac{1}{2}|\nabla_\hor f|^2\right).
\end{split}
\]
\end{proof}

\smallskip

\section{Proof of the main results}

This section is devoted to the proofs of the main results. We begin with

\begin{lem}\label{mainlem1}
Assume that the orthogonal complement $D^\perp$ of the distribution $D$ is involutive and is given by the span of $n-k$ horizontal isometries. Then the followings hold:
\begin{enumerate}
\item $\frac{d}{dt}(f_t(\varphi_t))=-\dot f_t(\varphi_t)+2\Delta_\hor f_t(\varphi_t)+ 2V(\varphi_t) + 2Kf_t(\varphi_t)$,
\item $\frac{d}{dt}(\dot f_t(\varphi_t)) =\Delta_\hor\dot f_t(\varphi_t)+K\dot f_t(\varphi_t)$,
\item $\frac{d}{dt}\left(\frac{1}{2}|\nabla_\ver f_t|^2_{\varphi_t}\right)=\Delta_\hor \left(\frac{1}{2}|\nabla_\ver f_t|^2\right)(\varphi_t)\\
-\sum_i\left|(\nabla_{\Rb_i}\nabla_\hor f_t-\nabla_{\nabla_\hor f_t}\Rb_i)_\hor\right|_{\varphi_t}^2+\left<\nabla_\ver f_t,\nabla V\right>_{\varphi_t}$.
\end{enumerate}
\end{lem}

\begin{proof}[Proof of Lemma \ref{mainlem1}]
The first assertion follows from (\ref{eqn}). By (\ref{eqn}), we have
\[
\begin{split}
&\frac{d}{dt}(\dot f_t(\varphi_t))=\frac{d}{dt}\left(-\frac{1}{2}|\nabla_\hor f_t|^2_{\varphi_t}+\Delta_\hor f_t(\varphi_t)+V(\varphi_t)+Kf_t(\varphi_t)\right)\\
& =-\left<\nabla_\hor f_t,\nabla_\hor\dot f_t\right>_{\varphi_t}-\frac{1}{2}\left<\nabla|\nabla_\hor f_t|^2,\nabla_\hor f_t\right>_{\varphi_t}+\Delta_\hor\dot f_t(\varphi_t) \\
&\quad +\left<\nabla_\hor\Delta_\hor f_t,\nabla_\hor f_t\right>+\left<\nabla V,\nabla_\hor f\right>+K\dot f_t(\varphi_t)+K|\nabla_\hor f_t|^2_{\varphi_t}\\
&  =\Delta_\hor\dot f_t(\varphi_t)+K\dot f_t(\varphi_t).
\end{split}
\]

By (1) and (8) of Lemma \ref{facts}, we have
\[
\begin{split}
\left<\nabla_\hor (\Rb_i f),\nabla_\hor f\right>&=\left<\nabla_{\Rb_i}\nabla_\hor f-\nabla_{\nabla_\hor f}\Rb_i,\nabla_\hor f\right>\\
&=\left<\nabla\left(\frac{1}{2}|\nabla_\hor f|^2\right),\Rb_i\right>.
\end{split}
\]

Therefore, by combining this with (\ref{eqn}), we have
\[
\begin{split}
\frac{d}{dt}\left(\frac{1}{2}|\nabla_\ver f_t|^2_{\varphi_t}\right)&=\sum_i \Rb_if_t(\varphi_t)\frac{d}{dt}(\Rb_if_t(\varphi_t))\\
&=\sum_i\Rb_if_t(\varphi_t)\left(\Rb_i\dot f_t(\varphi_t)+\left<\nabla_\hor f_t,\nabla_\hor (\Rb_if_t)\right>_{\varphi_t}\right)\\
&=\sum_i\Rb_if_t(\varphi_t)\Rb_i(\Delta f_t+ V + Kf_t)(\varphi_t)\\
&=\left<\nabla_\ver f_t,\nabla\Delta_\hor f_t+\nabla V\right>_{\varphi_t}+K|\nabla_\ver f|^2_{\varphi_t}.
\end{split}
\]

Finally, by (10) of Lemma \ref{facts}, we obtain
\[
\begin{split}
&\frac{d}{dt}\left(\frac{1}{2}|\nabla_\ver f_t|^2_{\varphi_t}\right)=\Delta_\hor \left(\frac{1}{2}|\nabla_\ver f_t|^2\right)(\varphi_t)+K|\nabla_\ver f_t|_{\varphi_t}^2\\
&-\sum_i\left|(\nabla_{\Rb_i}\nabla_\hor f_t-\nabla_{\nabla_\hor f_t}\Rb_i)_\hor\right|_{\varphi_t}^2+\left<\nabla_\ver f_t,\nabla V\right>_{\varphi_t}.
\end{split}
\]
\end{proof}

\begin{proof}[Proof of Theorem \ref{keythm}]
By Lemma \ref{keylem} and (12) of Lemma \ref{facts},
\[
\begin{split}
\frac{d}{dt}\Delta_\hor f_t(\varphi_t)&\leq \frac{a_4(t)^2}{k}-\frac{2a_4(t)}{k}\Delta_\hor f_t(\varphi_t)+ \Delta_\hor\left(\dot f_t+\frac{1}{2}|\nabla_\hor f|^2\right)(\varphi_t)\\
&-\ric^\hor_{\varphi_t}(\nabla f_t,\nabla f_t)-3\ric^\ver_{\varphi_t}(\nabla_\hor f_t, \nabla_\hor f_t)\\
&-4\sum_k\left<\nabla_{\nabla_\hor f_t}\Rb_k,\nabla_{\Rb_k}\nabla_\hor f_t-\nabla_{\nabla_\hor f_t}\Rb_k\right>_{\varphi_t}.
\end{split}
\]

By Young's inequality and (12) of Lemma \ref{facts}, the above inequality becomes
\[
\begin{split}
&\frac{d}{dt}\Delta_\hor f_t(\varphi_t) \leq \frac{a_4(t)^2}{k}-\frac{2a_4(t)}{k}\Delta_\hor f_t(\varphi_t)+ \Delta_\hor\left(\dot f_t+\frac{1}{2}|\nabla_\hor f|^2\right)(\varphi_t)\\
&+a_2(t)\sum_k|(\nabla_{\nabla_\hor f_t}\Rb_k-\nabla_{\Rb_k}\nabla_\hor f_t)_\hor|^2_{\varphi_t}+\frac{4}{a_2(t)}\ric^\ver_{\varphi_t}(\nabla_\hor f_t, \nabla_\hor f_t)\\
&-\ric^\hor_{\varphi_t}(\nabla f_t,\nabla f_t)-3\ric^\ver_{\varphi_t}(\nabla_\hor f_t, \nabla_\hor f_t).
\end{split}
\]

Using the first and the second assumptions,
\[
\begin{split}
\frac{d}{dt}\Delta_\hor f_t(\varphi_t) &\leq \frac{a_4(t)^2}{k}-\frac{2a_4(t)}{k}\Delta_\hor f_t(\varphi_t)+  \Delta_\hor\left(\dot f_t+\frac{1}{2}|\nabla_\hor f|^2\right)(\varphi_t)\\
&\quad +a_2(t)\sum_k|(\nabla_{\nabla_\hor f_t}\Rb_k-\nabla_{\Rb_k}\nabla_\hor f_t)_\hor|^2_{\varphi_t}\\
&\quad +\frac{4K_3}{a_2(t)}|\nabla_\hor f_t|^2_{\varphi_t}-K_1|\nabla_\hor f_t|^2_{\varphi_t} -K_2|\nabla_\ver f_t|^2_{\varphi_t}.
\end{split}
\]

Let $F_t(x)=\Delta_\hor f_t(x)+a_1(t)\dot f_t(x)+\frac{a_2(t)}{2}|\nabla_\ver f_t(x)|^2+a_3(t)f_t(x)$. By Lemma \ref{mainlem1}, we have
\[
\begin{split}
&\frac{d}{dt}F_t(\varphi_t) \leq \frac{a_4(t)^2}{k}-\frac{2a_4(t)}{k}\Delta_\hor f_t(\varphi_t)+  \Delta_\hor\left(\dot f_t+\frac{1}{2}|\nabla_\hor f|^2\right)(\varphi_t)\\
& +\frac{4K_3}{a_2(t)}|\nabla_\hor f_t|^2_{\varphi_t}-K_1|\nabla_\hor f_t|^2_{\varphi_t} -K_2|\nabla_\ver f_t|^2_{\varphi_t} +\dot a_1(t)\dot f_t(\varphi_t) +\frac{\dot a_2(t)}{2}|\nabla_\ver f_t|^2_{\varphi_t}\\
&+\dot a_3(t)f_t(\varphi_t) +a_1(t)\Delta_\hor\dot f_t(\varphi_t)+a_1(t) K\dot f_t(\varphi_t)+a_2(t)\Delta_\hor \left(\frac{1}{2}|\nabla_\ver f_t|^2\right)(\varphi_t)\\
& +a_2(t)\left<\nabla_\ver f_t,\nabla V\right>_{\varphi_t}+a_3(t)(-\dot f_t(\varphi_t)+2\Delta_\hor f_t(\varphi_t)+ 2V(\varphi_t) + 2Kf_t(\varphi_t)).
\end{split}
\]

By (\ref{eqn}) and collecting terms, we obtain
\[
\begin{split}
&\frac{d}{dt}F_t(\varphi_t) \leq \frac{a_4(t)^2}{k}+  \Delta_\hor F_t+ \Delta_\hor V+2a_3(t)V(\varphi_t)\\
& +2\left(\frac{4K_3}{a_2(t)}-K_1\right)V+\left(\frac{\dot a_2(t)}{2}-K_2\right)|\nabla_\ver f_t|^2_{\varphi_t}\\
&+\left(a_3(t)-\frac{2a_4(t)}{k} + K+2\left(\frac{4K_3}{a_2(t)}-K_1\right)\right)\Delta_\hor f_t\\
&+\left(\dot a_1(t)+a_1(t) K-a_3(t)-2\left(\frac{4K_3}{a_2(t)}-K_1\right)\right)\dot f_t(\varphi_t) \\
&+\left(\dot a_3(t)+2Ka_3(t)+2K\left(\frac{4K_3}{a_2(t)}-K_1\right)\right)f_t(\varphi_t)\\
& +a_2(t)\left<\nabla_\ver f_t,\nabla V\right>_{\varphi_t}.
\end{split}
\]

By Young's inequality and the definition of $F_t$, we obtain
\[
\begin{split}
&\frac{d}{dt}F_t(\varphi_t) \leq \frac{a_4(t)^2}{k}+  \Delta_\hor F_t(\varphi_t)+ \Delta_\hor V(\varphi_t)+2a_3(t)V(\varphi_t)\\
&+K_5|\nabla_\ver V|^2_{\varphi_t}+\left(a_3(t)-\frac{2a_4(t)}{k} + K+\frac{8K_3}{a_2(t)}-2K_1\right)F_t(\varphi_t)\\
& +2\left(\frac{4K_3}{a_2(t)}-K_1\right)V(\varphi_t)+\left(\frac{\dot a_2(t)}{2}-K_2-4K_3\right)|\nabla_\ver f_t|^2_{\varphi_t}\\
&+a_2(t)\left(\frac{a_2(t)}{4K_5}-\frac{a_3(t)}{2}+\frac{a_4(t)}{k} - \frac{K}{2}+K_1\right)|\nabla_\ver f_t|^2_{\varphi_t}\\
&+\left(\dot a_1(t)+\frac{2a_1(t)a_4(t)}{k}+(a_1(t)+1)\left(2K_1-a_3(t) -\frac{8K_3}{a_2(t)}\right)\right) \dot f_t(\varphi_t)\\
&+\left(\dot a_3(t)+\frac{2a_3(t)a_4(t)}{k}+(K-a_3(t))\left(a_3(t)+\frac{8K_3}{a_2(t)}-2K_1\right)\right)f_t(\varphi_t).
\end{split}
\]

By assumptions (6), (7), and (8), the inequality becomes
\[
\begin{split}
\frac{d}{dt}F_t(\varphi_t) &\leq \frac{a_4(t)^2}{k}+  \Delta_\hor F_t(\varphi_t)+ \Delta_\hor V(\varphi_t)+K_5|\nabla_\ver V|^2_{\varphi_t}\\
&+\left(a_3(t)-\frac{2a_4(t)}{k} + K+\frac{8K_3}{a_2(t)}-2K_1\right)F_t(\varphi_t)\\
& +2\left(a_3(t)+\frac{4K_3}{a_2(t)}-K_1\right)V(\varphi_t).
\end{split}
\]

By assumptions (3), (4), and (5),
\[
\begin{split}
\frac{d}{dt}F_t(\varphi_t) &\leq \frac{a_4(t)^2}{k}+  \Delta_\hor F_t(\varphi_t)+K_6+2\left(a_3(t)+\frac{4K_3}{a_2(t)}\right)K_4\\
&+\left(a_3(t)-\frac{2a_4(t)}{k} + K+\frac{8K_3}{a_2(t)}-2K_1\right)F_t(\varphi_t).
\end{split}
\]

Let $\epsilon>0$ and let  $r_\epsilon(\cdot)$ be a solution of
\[
\begin{split}
\dot r_\epsilon(t) &= \frac{a_4(t)^2}{k} +K_6+2\left(a_3(t)+\frac{4K_3}{a_2(t)}\right)K_4\\
&+\left(a_3(t)-\frac{2a_4(t)}{k} + K+\frac{8K_3}{a_2(t)}-2K_1\right)r_\epsilon(t)+\epsilon.
\end{split}
\]
with condition $r_\epsilon (t)\to \infty$ as $t\to 0$.

Let $t_0>0$ be the first time where there is a point $x$ in $M$ satisfying $F_{t_0}(\varphi_{t_0}(x))=r(t_0)$. Then
\[
\begin{split}
\dot r_\epsilon(t_0)&\leq \frac{d}{dt}F_t(\varphi_t)\leq \frac{a_4(t_0)^2}{k}+ K_6+2\left(a_3(t_0)+\frac{4K_3}{a_2(t_0)}\right)K_4\\
&+\left(a_3(t_0)-\frac{2a_4(t_0)}{k} + K+\frac{8K_3}{a_2(t_0)}-2K_1\right)r_\epsilon(t_0)<\dot r_\epsilon(t_0)
\end{split}
\]
which is a contradiction.

Therefore, $F_t(x)< r_\epsilon(t)$ for all $t\geq 0$ and all $x$ in $M$. The result follows from stability of $r$.
\end{proof}

\begin{proof}[Proof of Corollary \ref{keycor1}]
This follows from Theorem \ref{keythm} by setting $K_1=K=0$, $a_3\equiv 0$, $a_1\equiv c$, $a_4(t)=\frac{4(c+1)kK_3}{ca_2(t)}$, $a_2(t)=2\left(K_2-\frac{4K_3}{c}\right)t$, and $r(t)=\frac{4(c+1)^2kK_3^2}{(cK_2-4K_3)(8K_3-cK_2)t}$.
\end{proof}

\begin{proof}[Proof of Corollary \ref{keycor2}]
If we set $K_1=K=0$, $a_3\equiv 0$, $a_1\equiv c$ is a constant, $a_2(t)=c_1\tanh(c_2t)$,  $a_4(t)=\frac{4k(c+1)K_3}{ca_2(t)}$, and $K_5=\frac{4k(a_1+1)^2K_3^2}{a_1K_6(8K_3-a_1K_2)}$ in Theorem \ref{keythm}. Then a computation shows that $r(t)=\frac{K_6}{c_2}\coth(c_2t)+cK_4$ and the result follows from Theorem \ref{keythm}
\end{proof}

\smallskip

\section{Appendix}

In this appendix, we provide the detail calculations that we used in the proof of Theorem \ref{keythm}.

\begin{lem}\label{facts}
Suppose that the assumptions in Lemma \ref{keylem} hold. Let $\Rb_1,...,\Rb_{n-k}$ be the horizontal isometries which span $D^\perp$. Let $X_1$ and $X_2$ be vector fields contained in $D$ and let $Z$ be a vector field contained in $D^\perp$. Then the followings hold:
\begin{enumerate}
\item $\left<\nabla_{X_1} Z,X_2\right>=-\left<\nabla_{X_2} Z,X_1\right>$,
\item $\nabla_{v_i}v_i$ is horizontal,
\item $\nabla_{\Rb_i}\Rb_j$ is vertical,
\item $\nabla_{v_j}\Rb_i$ is horizontal,
\item $\nabla_{\nabla_\hor f}\nabla_\hor f =\frac{1}{2}\nabla_\hor|\nabla_\hor f|^2- 2(\nabla_{\nabla_\hor f}\nabla_\ver f)_\hor$,
\item $\sum_j\left<\nabla_{[\Rb_i,v_j]}\nabla f, v_j\right>=0$,
\item $\sum_j\left<\nabla_{[\Rb_i,v_j]}\nabla_\hor f, v_j\right>=-\sum_j\left<\nabla_{v_j}\nabla_\hor f, [\Rb_i,v_j]\right>$,
\item $\nabla_\hor (\Rb_i f) = (\nabla_{\Rb_i}\nabla_\hor f-\nabla_{\nabla_\hor f}\Rb_i)_\hor$,
\item $\Delta_\hor (\Rb_i f)=\Rb_i(\Delta_\hor f)$,
\item $\Delta_\hor\left(\frac{1}{2}|\nabla_\ver f|^2\right)=\left<\nabla_\ver f,\nabla\Delta_\hor f\right>+\sum_i |(\nabla_{\Rb_i}\nabla_\hor f-\nabla_{\nabla_\hor f}\Rb_i)_\hor|^2$,
\item $\ric^\hor(\nabla_\ver f,\nabla_\ver f)=\sum_{j}|(\nabla_{v_j}\nabla_\ver f)_\hor|^2$,
\item $\ric^\hor(\nabla f,\Rb_i)=\sum_j\left<\nabla_{v_j}\nabla_{\nabla_\hor f}\Rb_i,v_j\right>$,
\item $\left<\Rm(v_j,\Rb_i)\Rb_j,v_j\right>=\left<\nabla_{v_j}\Rb_i,\nabla_{v_j}\Rb_j\right>$.
\end{enumerate}
\end{lem}

\begin{proof}
By assumption, we have
\[
\left<\nabla_{\Rb_i} X_1, X_2\right>+\left<\nabla_{\Rb_i}  X_2, X_1\right>=\Rb_i\left<X_1,X_2\right>=\left<[\Rb_i, X_1],X_2\right>+\left<[\Rb_i, X_2],X_1\right>.
\]
This gives (1).

It follows from (1) that
\[
\left<\nabla_{v_i}v_i,w_j\right>=-\left<v_i,\nabla_{v_i}w_j\right>=0
\]
which is (2).

Since $D$ is involutive, we also have
\[
\begin{split}
&\left<\nabla_{\Rb_i}\Rb_j,v_k\right>=-\left<\Rb_j,\nabla_{\Rb_i}v_k\right>=-\left<\Rb_j,\nabla_{v_k}\Rb_i\right>\\
&=\left<\nabla_{v_k}\Rb_j,\Rb_i\right>=\left<\nabla_{\Rb_j}v_k,\Rb_i\right>=-\left<v_k,\nabla_{\Rb_j}\Rb_i\right>=-\left<v_k,\nabla_{\Rb_i}\Rb_j\right>.
\end{split}
\]
This gives (3).

Since
\[
\left<\nabla_{v_j}\Rb_i,\Rb_k\right> = \left<\nabla_{\Rb_i}v_j,\Rb_k\right> = -\left<v_j,\nabla_{\Rb_i}\Rb_k\right> =0,
\]
(4) holds.

The statement (5) follows from
\[
\begin{split}
\left<\frac{1}{2}\nabla|\nabla_\hor f|^2,X_1\right> &= \left<\nabla_{X_1}\nabla_\hor f,\nabla_\hor f\right>\\
&=  \left<\nabla_{X_1}\nabla f,\nabla_\hor f\right>- \left<\nabla_{X_1}\nabla_\ver f,\nabla_\hor f\right>\\
&=  \left<\nabla_{\nabla_\hor f}\nabla f,X_1\right>+ \left<\nabla_{\nabla_\hor f}\nabla_\ver f,X_1\right>\\
&=  \left<\nabla_{\nabla_\hor f}\nabla_\hor f,X_1\right>+ 2\left<\nabla_{\nabla_\hor f}\nabla_\ver f,X_1\right>.
\end{split}
\]

Since $\Rb_i$ is a horizontal isometry, $[\Rb_i,v_j]$ is horizontal. Therefore,
\[
\begin{split}
\sum_j\left<\nabla_{[\Rb_i,v_j]}\nabla f, v_j\right>&=\sum_j\left<\nabla_{v_j}\nabla f, [\Rb_i,v_j]\right>\\
&=\sum_{j,k}\left<\nabla_{v_j}\nabla f,v_k\right>\left<v_k, [\Rb_i,v_j]\right>\\
& =-\sum_{j,k}\left<\nabla_{v_j}\nabla f,v_k\right>\left<v_j, [\Rb_i,v_k]\right>=0
\end{split}
\]
which is (6).

It follows from (1) and (6) that
\[
\begin{split}
0=\sum_j\left<\nabla_{[\Rb_i,v_j]}\nabla f, v_j\right>&=\sum_j\left<\nabla_{v_j}\nabla f, [\Rb_i,v_j]\right>\\
&=\sum_j\left<\nabla_{v_j}\nabla_\hor f, [\Rb_i,v_j]\right> +\left<\nabla_{v_j}\nabla_\ver f, [\Rb_i,v_j]\right>\\
&=\sum_j\left<\nabla_{v_j}\nabla_\hor f, [\Rb_i,v_j]\right> -\left<\nabla_{[\Rb_i,v_j]}\nabla_\ver f, v_j\right>.
\end{split}
\]

By (6), we also have
\[
0=\sum_j\left<\nabla_{[\Rb_i,v_j]}\nabla_\hor f, v_j\right>+\sum_j\left<\nabla_{[\Rb_i,v_j]}\nabla_\ver f, v_j\right>.
\]
Therefore, (7) follows.

Statement (8) follows from
\[
\begin{split}
\left<\nabla (\Rb_i f),v_j\right> &= \left<\nabla_{v_j}\nabla f, \Rb_i\right> +\left<\nabla_\hor f, \nabla_{v_j}\Rb_i\right>+\left<\nabla_\ver f, \nabla_{v_j}\Rb_i\right>\\
 &= \left<\nabla_{\Rb_i}\nabla f,v_j\right> +\left<\nabla_\hor f, \nabla_{v_j}\Rb_i\right>+\left<\nabla_\ver f, \nabla_{\Rb_i}v_j\right>\\
&= \left<\nabla_{\Rb_i}\nabla_\hor f, v_j\right> -\left<\nabla_{\nabla_\hor f}\Rb_i,v_j\right>.
\end{split}
\]

By (1) and (8), we have
\[
\begin{split}
&\Delta_\hor (\Rb_i f)=\sum_j\left<\nabla_{v_j}(\nabla_{\Rb_i}\nabla_\hor f-\nabla_{\nabla_\hor f}\Rb_i),v_j\right>\\
&=\sum_j\left<\nabla_{\Rb_i}\nabla_{v_j}\nabla_\hor f,v_j\right>-\left<\nabla_{\nabla_\hor f}\nabla_{v_j}\Rb_i,v_j\right>\\
&+\sum_j\left<\nabla_{[v_j,\Rb_i]}\nabla_\hor f-\nabla_{[v_j,\nabla_\hor f]}\Rb_i,v_j\right>.
\end{split}
\]

By (1), (3), and (4), the above equation becomes
\[
\begin{split}
&\Delta_\hor (\Rb_i f)=\sum_j\left<\nabla_{\Rb_i}\nabla_{v_j}\nabla_\hor f,v_j\right>-\left<\nabla_{\nabla_\hor f}\nabla_{v_j}\Rb_i,v_j\right>\\
&+\sum_j\left<\nabla_{[v_j,\Rb_i]}\nabla_\hor f,v_j\right>+\left<\nabla_{v_j}\Rb_i,\nabla_{v_j}\nabla_\hor f\right>-\left<\nabla_{v_j}\Rb_i,\nabla_{\nabla_\hor f}v_j\right>.
\end{split}
\]

By (1) and (7),
\[
\begin{split}
\Delta_\hor (\Rb_i f)&=\sum_j\left<\nabla_{\Rb_i}\nabla_{v_j}\nabla_\hor f,v_j\right>\\
&-\sum_j\left<\nabla_{v_j}\nabla_\hor f,[v_j,\Rb_i]\right>+\left<\nabla_{v_j}\Rb_i,\nabla_{v_j}\nabla_\hor f\right>\\
&=\sum_j\left<\nabla_{\Rb_i}\nabla_{v_j}\nabla_\hor f,v_j\right>+\sum_j\left<\nabla_{v_j}\nabla_\hor f,\nabla_{\Rb_i}v_j\right>.
\end{split}
\]
This, together with (6), gives (9).

By (8) and (9), we have
\[
\begin{split}
\Delta_\hor\left(\frac{1}{2}|\nabla_\ver f|^2\right)&= \sum_{i,j}(v_i\Rb_j f)^2+\sum_j\Rb_j f \Delta_\hor(\Rb_j f)\\
&= \sum_{j}|\nabla_\hor(\Rb_j f)|^2+\left<\nabla_\ver f,\nabla \Delta_\hor f\right>\\
&= \sum_{j}|(\nabla_{\Rb_i}\nabla_\hor f-\nabla_{\nabla_\hor f}\Rb_i)_\hor|^2+\left<\nabla_\ver f,\nabla \Delta_\hor f\right>
\end{split}
\]
which is (10).

By (3), $\nabla_{\nabla_\ver f}\nabla_\ver f$ is vertical. Therefore, by (1), we have
\[
\begin{split}
\ric^\hor(\nabla_\ver f,\nabla_\ver f)=-\sum_j\left<\nabla_{\nabla_\ver f}\nabla_{v_j}\nabla_\ver f,v_j\right>-\sum_j\left<\nabla_{[v_j,\nabla_\ver f]}\nabla_\ver f,v_j\right>.
\end{split}
\]

By (1), the above becomes
\[
\begin{split}
&\ric^\hor(\nabla_\ver f,\nabla_\ver f)=\sum_j\left(\left<\nabla_{v_j}\nabla_\ver f,\nabla_{\nabla_\ver f} v_j\right>+\left<\nabla_{v_j}\nabla_\ver f,[v_j,\nabla_\ver f]_\hor\right>\right)\\
&=\sum_j\left(\left<\nabla_{v_j}\nabla_\ver f,(\nabla_{\nabla_\ver f} v_j)_\ver\right>+|(\nabla_{v_j}\nabla_\ver f)_\hor|^2\right).
\end{split}
\]
It follows from (4) that $\nabla_{\nabla_\ver f} v_j$ is horizontal. Therefore, (11) holds.

By (1) and (3),
\[
\begin{split}
\ric^\hor(\nabla f,\Rb_i)&=\sum_j\left<\nabla_{v_j}\nabla_{\nabla_\hor f}\Rb_i,v_j\right>+\left<\nabla_{v_j}\Rb_i,\nabla_{\nabla f}v_j\right>-\left<\nabla_{[v_j,\nabla f]}\Rb_i,v_j\right>\\
&=\sum_j\left<\nabla_{v_j}\nabla_{\nabla_\hor f}\Rb_i,v_j\right>+\left<\nabla_{v_j}\Rb_i,\nabla_{\nabla f}v_j\right>+\left<\nabla_{v_j}\Rb_i,[v_j,\nabla f]\right>\\
&=\sum_j\left<\nabla_{v_j}\nabla_{\nabla_\hor f}\Rb_i,v_j\right>+\left<\nabla_{v_j}\Rb_i,\nabla_{v_j}\nabla f\right>.
\end{split}
\]

Statement (12) follows since
\[
\sum_j\left<\nabla_{v_j}\Rb_i,\nabla_{v_j}\nabla f\right>=\sum_{j,k}\left<\nabla_{v_j}\Rb_i,v_k\right>\left<v_k,\nabla_{v_j}\nabla f\right>=0
\]
by (1).

By (1) and (3), we have
\[
\begin{split}
\left<\Rm(v_j,\Rb_i)\Rb_k,v_j\right>&=\left<\nabla_{v_j}\Rb_k,\nabla_{\Rb_i}v_j\right>+\left<\nabla_{v_j}\Rb_k,[v_j,\Rb_i]\right>\\
&=\left<\nabla_{v_j}\Rb_i,\nabla_{v_j}\Rb_k\right>
\end{split}
\]
which is (13).
\end{proof}

\end{document}